\documentclass{article}

\usepackage{amsthm}
\usepackage{amsfonts}
\usepackage{amsmath}
\usepackage{mathtools}
\usepackage{graphicx}
\usepackage{epstopdf}
\usepackage{algorithm}
\usepackage{algpseudocode}
\usepackage{bm}
\newtheorem*{ass}{Assumption}
\newtheorem*{theorem}{Theorem}
\newtheorem{lemma}{Lemma}
\newtheorem{corollary}{Corollary}

\newcommand{\D}{\mathbf{D}^{0}}
\newcommand{\Dp}{\mathbf{D}^{+}}
\newcommand{\Dm}{\mathbf{D}^{-}}
\newcommand{\x}{\mathbf{x}}
\newcommand{\e}{\mathbf{e}}
\newcommand{\y}{\mathbf{y}}

\newcommand{\ue}{u}
\newcommand{\ueb}{\mathbf{u}}
\newcommand{\pe}{p}

\newcommand{\feb}{\mathbf{f}}

\newcommand{\Xeb}{\mathbf{X}}

\newcommand{\Ueb}{\mathbf{U}}
\newcommand{\F}{\mathbf{F}}
\newcommand{\dc}{\delta_{c}}

\newcommand{\length}{L}

\newcommand{\Dom}{\Omega}
\newcommand{\Mu}{M_{\ueb}}
\newcommand{\MDu}{M_{\frac{\partial \ueb}{\partial \x}}}
\newcommand{\MF}{M_{\F}}
\newcommand{\LF}{L_{\F}}
\newcommand{\Md}{M_{\dc}}
\newcommand{\Ld}{L_{\dc}}

\newcommand{\uc}{\tilde{u}}
\newcommand{\ucb}{\tilde{\mathbf{u}}}
\newcommand{\wcb}{\mathbf{w}}
\newcommand{\wc}{w}
\newcommand{\pc}{\tilde{p}}

\newcommand{\Xcb}{\tilde{\mathbf{X}}}

\newcommand{\LexpM}{\mathbf{\Lambda}}
\newcommand{\Lexp}{\Lambda}

\newcommand{\dx}{h}
\newcommand{\dt}{\Delta t}
\newcommand{\Grid}{G_{h}}
\newcommand{\Muc}{M_{\uc}}

\newcommand{\uEb}{\mathbf{v}}
\newcommand{\uE}{v}
\newcommand{\usEb}{\mathbf{w}}

\newcommand{\upE}{\psi}
\newcommand{\XEb}{\mathbf{Z}}
\newcommand{\pE}{q}
\newcommand{\ures}{\bm{\tau}}
\newcommand{\pres}{\eta}
\newcommand{\Xres}{\bm{\xi}}

\title{Convergence of the Immersed Boundary Method for an Elastically Bound Particle Immersed in a 2D
 Navier-Stokes Fluid Fluid}

\author{Alexandre X. Milewski\thanks{Courant Institute of Mathematical Sciences, New York, NY.} 
  \and Charles S. Peskin\footnotemark[2]
  }
  \date{October 6th, 2025}

\begin{document}

\maketitle

\begin{abstract}
The immersed boundary (IB) method has been used as a means to simulate fluid-membrane interactions in a wide variety of biological and engineering applications. Although the numerical convergence of the method has been empirically
verified, it is theoretically unproved because of the singular
forcing terms present in the governing equations. This paper is motivated by a specific variant of the IB method, in which the fluid is 2 dimensions greater than the dimension of the immersed structure. In these co-dimension 2 problems the immersed boundary is necessarily mollified in the continuous formulation. In this paper we leverage this fact to prove convergence of the IB method as applied to a moving elastically bound particle in a fully non-linear fluid.
\end{abstract}

\section{Introduction}
\label{sec:Intro}
Fluid structure interactions (FSI) consist of a large class of problems, whose computation is particularly challenging when the structure is deformable. FSI problems arise in a broad range of scientific and engineering applications. Some notable examples are bio-mechanics problems such as intracardiac flow \cite{davey2024simulating}, lamprey locomotion \cite{hamlet2018role}, blood-cell deformation \cite{fai2017image}; and engineering problems such as unsteady aerodynamics \cite{haase2001unsteady} and turbulent FSI simulations \cite{ikeno2007finite}

One approach to solving FSI problems are conforming mesh methods, in which new fluid and structural meshes are generated at every timestep to conform with the evolving fluid structure interface. These methods tend to give a high fidelity representation of the physics near the fluid structure interface, but they suffer from the need to regenerate a mesh at every timestep.

The immersed boundary (IB) method takes a different approach: the fluid mesh is described on a static Eulerian grid, and the immersed structure is represented by a network of points on a an evolving secondary mesh. The immersed structure influences the fluid through forces: the potential energy of a given structure configuration determines (via the variational derivative) the distribution of forces applied to fluid, which in turn drives the fluid motion. The fluid then advects (and deforms) the immersed structure with the fluid velocity.

Since the structure mesh and the fluid mesh do not in general coincide, there is a need to transmit the data from the structure mesh to the fluid mesh and vice versa. This task is accomplished through the use of a regularized approximation to the Dirac delta function centred on each point of the immersed structure: Forces on the immersed points act on the fluid mesh points in proportion to the amplitude of the delta function approximant at that mesh point; likewise, the fluid velocity at the nearby fluid gridpoints contribute to the velocity of the immersed mesh point in proportion to the amplitude of the delta function approximant evaluated at the immersed mesh point. As the numerical meshes are refined, the approximate delta function converges to a Dirac delta function. The approximating delta function can be chosen such that the IB scheme preserves mechanical quantities such as momentum and angular momentum.  Energy conservation is also ensured for the time-continuous spacially discretised case, regardless of the choice of regularized delta function,
 provided that the \textit{same} choice is made for velocity interpolation
and force spreading.

Since its inception in the 1970's (see \cite{peskin1972flow}), the method has since been extended with many variations, and has developed into a general methodology for a large range of fluid structure interactions. Some of these variants include: penalty functions to simulate immersed structures with inertia \cite{kim2007penalty}, or to enforce rigid body constraints \cite{kim2016penalty}; models that allow an immersed fibre to resist torsion \cite{lim2008dynamics}; and stochastic formulations that account for thermal fluctuations in the fluid, for simulations on a microscopic length-scale \cite{tabak2015stochastic}.

Fluid structure interaction problems can be classified by comparing the dimension of the structure to that of the ambient fluid. Those dimensions can be the same, as in the case of a 3 dimensional body in a 3 dimensional fluid or a 2 dimensional body in a 2 dimensional fluid; or they can differ by one, as in the case of a 2 dimensional membrane in a 3 dimensional fluid or a 1 dimensional membrane in a 2 dimensional fluid; or they can differ by more than one, as in the cases of a 1 dimensional fibre in a 3 dimensional fluid or a particle in a 3 or 2 dimensional fluid. Some examples of these different classifications being applied in practice are: McQueen and Peskin's simulation of the heart \cite{mcqueen2001heart} has co-dimension 0 for the thick muscular heart walls and co-dimension 1 for the thin valve leaflets\footnote{A more recent IB model for the human heart \cite{davey2024simulating} uses co-dimension 0 for both the the heart walls and the valve leaflets}; Lai and Peskin's \cite{lai2000immersed} simulation of 2 dimensional flow past a cylinder has co-dimension 1; and Han and Peskin's model of synchrony breaking cilia \cite{han2018spontaneous} has co-dimension 2 as the cilia are represented by space curves.

In the co-dimension 2 case, there is a subtlety in that a co-dimension 2 structure (e.g., an infinitely thin fibre in a three-dimensional fluid) cannot exert a drag force on the fluid. It is therefore necessary to somehow incorporate the width of the fibre into the model as a parameter. In the IB method, this is done by the choice of delta function. In other co-dimensional cases the delta function approximant would typically converge to the Dirac delta function, but in the co-dimension 2 case, the width of the approximating delta function remains unchanged as the computational meshes are refined. Thus, even in the continuous model, there is a smooth transition between where the fluid ends and the fibre begins. This makes the numerical scheme in the co-dimension 2 case particularly amenable to analysis, as we shall see.

In the study of IB methods, a question of theoretical interest is the proof of numerical convergence of the discretised equations to the underlying continuous model they approximate. In a typical numerical convergence proof, the argument runs as follows: assuming the solution to the continuous model is sufficiently regular, calculate its residues when the exact solution is substituted into the finite difference scheme. If these residues are uniformly convergent in the limit as the numerical parameters (like gridwidth and timestep size) go to zero, we say the numerical method is consistent. Additionally, a method is stable, if the errors can be bounded by the residues. Lax and Richtmyer's landmark paper `{\it Survey of the stability of linear finite Difference equations}' (see \cite{lax2005survey}) demonstrates that consistent, stable schemes are convergent. Although Lax and Richtmyer's original paper was written in the context of linear PDE's, the same arguments can be extended to the non-linear case. Proofs of this kind have been used since the 1960's to prove numerical convergence of Navier-Stokes solvers, see \cite{chorin1969convergence} as an example.

The same method of proof does not go through as readily in the case of the immersed boundary method. The difficulty is related to the singularity of the Dirac delta function. Progress was made by Y. Mori (2008) \cite{mori2008convergence}, who proved numerical convergence of an IB solver for the velocity field generated by a given distribution of force along a one-dimensional immersed boundary in a two-dimensional Stokes fluid.  This proof involves the velocity field at a single instant of time, and motion of the immersed boundary is not considered. Numerical convergence to the fully coupled equations is still unproved, despite ample empirical evidence that the method is indeed convergent---cf. \cite{lai2000immersed} or \cite{griffith2012simulating}.

In this paper we prove convergence of the IB method applied to the
 case of a co-dimension 2 immersed boundary.  The distinctive feature
 of this case is that the Dirac delta function is already mollified at
 the level of model formulation, since the width of the mollified
 delta function is a physical parameter that represents the small but
 non-zero width of the immersed particle or fibre.  This resolves
 the singularity issue and opens the door to a convergence proof that
 follows the approach pioneered by Peter Lax and Robert Richtmyer.

The remainder of this paper is divided into 4 sections: In \ref{sec:Formulation}, we outline a simple co-dimension 2 IB problem, we describe the numerical scheme, and we state the convergence theorem that we aim to prove. We also state the lemmas (such as consistency and the discrete Helmholtz decomposition) that we will need to prove the theorem. In \ref{sec:proof}, we use the supporting lemmas to prove the convergence theorem. In \ref{sec:lem} we go back and prove the lemmas. And in \ref{sec:sim} we run simulations confirming the results of the theorem.

\section{Problem formulation and main theorem} \label{sec:Formulation}
\subsection{The immersed boundary particle problem} \label{ibproblem}

We consider a viscous incompressible fluid in a two-dimensional $\Dom = [0,\length]\times[0,\length]$ with periodic boundary conditions. The fluid has density $\rho$ and viscosity $\mu$. Its velocity and pressure will be denoted $\ueb(\x,t)$ and $\pe(\x,t)$. Immersed in the fluid is a neutrally buoyant particle, the position and velocity of which will be denoted by $\Xeb(t)$ and $\Ueb(t)$. Note that we use lowercase letters for the fluid variables which are functions of $\x$ and $t$; and uppercase letters for the particle variables, which are functions of $t$ only. The equations of motion are as follows

\begin{align}
 \rho\left(\frac{\partial\ueb}{\partial t} + (\ueb\cdot\nabla)\ueb \right) &= \mu\nabla^2 \ueb - \nabla \pe + \feb \label{ContMom} \\
\nabla\cdot\ueb & = 0 \label{ContIncomp} \\
\feb(\x,t) &= \F(\Xeb(t))\dc(\x-\Xeb(t)) \label{ContFD}\\
\frac{d\Xeb}{dt}(t) & =  \int_{\Dom} \ueb(\x,t)\dc(\x-\Xeb(t))d\x \label{ContInterp} 
\end{align}

Equations \eqref{ContMom}-\eqref{ContInterp} Navier-Stokes equations for a viscous incompressible fluid with an applied force density $\feb(\x,t)$, which in our case is defined by equation \eqref{ContFD}, in which $\F(\Xeb)$ is a given differentiable function that describes an external force acting on the immersed particle when its position is $\Xeb$. This force is applied to the fluid in the neighbourhood of the particle's current position $\Xeb(t)$. The function $\dc$, specifies how the force $\F(\Xeb(t))$ is distributed, or ``spread'' onto the fluid. To be physically consistent, we require that $\dc$ have a small support (the area of which is proportional to $c^2$) and integrate to 1; furthermore, to prove convergence of the IB method, we require that $\dc$ be at least twice differentiable\footnote{Typical IB kernels are only once continuously differentiable, but greater smoothness is certainly possible and is needed for some applications. For example, there are two ``Gaussian-like'' kernels constructed in \cite{bao2016gaussian} that have three continuous derivatives.}. With these combined properties, we may consider $\dc$ to be a regularized approximation to the Dirac delta function.

As discussed in the introduction, our immersed particle has co-dimension 2, and we have to regard $c$ as a physical parameter of the problem. Intuitively, $c$ is proportional to the diameter of the immersed particle. The limit $c\rightarrow 0$ is not a sensible one, since the application of a point force to a fluid results in an infinite velocity at the location where the force is applied.

In the following, we denote the vector components of $\ueb,\feb,\Xeb,\Ueb,$ etc... with subscripts (e.g. $\ueb=(u_{1},u_{2})$).
In the following discussion, we will make an assumption on the regularity of solutions to the continuous problem described above. This assumption will primarily be used to prove consistency of the numerical method, as is typical in finite difference proofs in the style of Lax and Richtmyer, but it will also be used to provide upper bounds on the fluid velocity and its spacial gradient, which are essential in dealing with the nonlinearity present in the Navier-Stokes equations.

\begin{ass}[Regularity of the Model]
For a given choice of initial conditions $\Xeb_{0}$ and $\ueb_{0}(\x)$ satisfying $\nabla \cdot \ueb_{0} = 0$, we assume the system of equations \eqref{ContMom}-\eqref{ContInterp} has a unique solution $\ueb,\nabla\pe,\Xeb$ on $\Dom$ for all time $0\leq t\leq T$, and that this solution is 4 times continuously differentiable. As a consequence of this, we may find constants $\Mu$ and $\MDu$ such that
\begin{equation}
|u_{\alpha}(\x,t)| \leq \Mu \qquad \mathrm{and} \qquad \left|\frac{\partial u_{\alpha}}{\partial x_{\beta}}(\x,t)\right| \leq \MDu \qquad
\end{equation}
For all $\x\in \Dom$, $0\leq t\leq T$, and $\alpha,\beta = 1,2$.
\end{ass}

From the assumptions made in the discussion above, we can define the following constants, which we assemble in table \ref{tbl:constants}.

\begin{table}[htbp]
{\footnotesize
  \caption{Constants which will be useful in proving convergence. Each inequality is assumed to hold for all $\x\in\Dom$, all $0\leq t \leq T$, and---if mentioned---for all $\alpha,\beta=1,2$.}  \label{tbl:constants}
\begin{center}
  \begin{tabular}{|c|c|} \hline
   Constant & \bf Definition \\ \hline
    $M_{\ueb}$ & $ |\ue_{\alpha}(\x,t)| \leq \Mu$  \\
         $\MDu$ & $\left|\frac{\partial \ue_{\alpha}}{\partial x_{\beta}}(\x,t)\right| \leq \MDu$ \\
         $\MF$ & $\|\F(\x)\| \leq \MF$ \\
         $\LF$ & $\quad \|\F(\x) - \F(\y)\| \leq \LF \|\x - \y\|$ \\
         $\Md$ & $\quad |\dc(\x)| \leq \Md$ \\
         $\Ld$ & $\quad |\dc(\x) - \dc(\y)| \leq \Ld\|\x - \y\|$ \\
         \hline
  \end{tabular}
\end{center}
}
\end{table}

\subsection{The Algorithm}\label{algorithm}

The numerical scheme we shall use to approximate \eqref{ContMom}-\eqref{ContInterp} is an Euler implicit-explicit scheme, with an implicit viscosity term and explicit advection and forcing terms. Let the fluid mesh $\Grid = \{(i\dx,j\dx):0\leq i,j \leq N-1\}$ be the set of Cartesian coordinates on a grid with gridwidth $\dx$, and let $\dt$ be the timestep size. The sequences of gridfunctions $\ucb^{n}(\x)$ and $\pc^{n}(\x)$ represent the fluid velocity and pressure respectively, evaluated at the coordinate $\x\in \Grid$, and a time $n\dt$. The sequence $\Xcb^{n}$ represents the position of the immersed particle at time $n\dt$. For a given initial condition, and a choice of $\dx$ and $\dt$, we can generate a sequence of gridfunctions which satisfy

\begin{align}
\phantom{\rho\left(\frac{\ucb^{n+1}-\ucb^{n}}{\dt} + (\ucb^{n}\cdot\D)\ucb^{n}\right)}
&\begin{aligned}
    \mathllap{\rho\left(\frac{\ucb^{n+1}-\ucb^{n}}{\dt} + (\ucb^{n}\cdot\D)\ucb^{n}\right)} &= -\D \pc^{n+1} \\
    &\qquad + \mu(\Dp\cdot\Dm)\ucb^{n+1} + \F(\Xcb^{n})\dc(\x-\Xcb^{n}) \label{DiscMom}
 \end{aligned} \\
\D \cdot \ucb^{n+1} &= 0 \label{DiscIncomp} \\
\frac{\Xcb^{n+1} - \Xcb^{n}}{\dt} &= \sum_{\x\in \Grid}\ucb^{n}(\x)\dc(\x-\Xcb^{n})\dx^2\label{DiscIBP}
\end{align}
Where $\D,\Dp$ and $\Dm$ are centred, forward, and backward difference approximations to $\nabla$ respectively. $\mathbf{D}\cdot$ refers to the corresponding discrete divergence. We note that $\D\cdot$ is the adjoint of $-\D$, that $\Dp\cdot$ is the adjoint of $-\Dm$, and vice versa. From these equations we obtain the following algorithm

\begin{algorithm}
\caption{Immersed Boundary Method}
\label{alg:thy}
\begin{algorithmic}
\While{$n < T/\dt$}
\State{Compute the term $\wcb^{n} = \ucb^{n}-\dt\ucb^{n}\cdot\D\ucb^{n} + \frac{\dt}{\rho}\F(\Xcb^{n})\dc(\x-\Xcb^{n})$}
\State{Solve $\begin{cases}
    \left(I-\dt\frac{\mu}{\rho}\Dp\cdot\Dm\right)\ucb^{n+1} = \wcb^{n} - \frac{\dt}{\rho}\D \pc^{n+1} \\
    \D \cdot \ucb^{n+1} = 0
\end{cases}$ for $\ucb^{n+1}$}
\State{Calculate $\Xcb^{n+1} = \Xcb^{n} + \dt \sum_{\x\in \Grid}\ucb^{n}(\x)\dc(\x-\Xcb^{n})\dx^2$}
\State{$n \leftarrow n+1$}
\EndWhile \\
\Return $\ucb^{n},\Xcb^{n}$
\end{algorithmic}
\end{algorithm}

Note that the solve for $\ucb^{n+1}$ is well defined. This can be seen by applying $\D\cdot$ to the momentum equation; since $\D,\Dp$, and $\Dm$ are commutative, we are left with
\begin{equation}
    \frac{\dt}{\rho}\D\cdot\D \pc^{n+1} = \D\cdot\wcb^{n}
\end{equation}
As the kernel of $\D\cdot\D$ (whose dimension is $1$ or $4$ depending on whether $N$ is odd or even) is equal\footnote{The fact that both kernels are equal follows from the adjointedness of $\D\cdot$ and $-\D$.} to the kernel of $\D$, there is a unique $\D\pc^{n+1}$ that is compatible the two equations of the solve step. Once the uniqueness of $\D\pc^{n+1}$ is established, what remains to be shown is that the operator $\Dp\cdot\Dm$ is semi-positive definite, which is true since $\Dp\cdot$ is the adjoint of $\Dm$.

\subsection{Convergence}

For the discrete and continuous systems described above, we prove the following

\begin{theorem}[Convergence]\label{thm:conv}
  Let $\ueb,\pe,\Xeb$ satisfy equations \eqref{ContMom}-\eqref{ContInterp}, along with the regularity assumptions specified above. Let $\ucb^{n},\pc^{n},\Xcb^{n}$ satisfy \eqref{DiscMom}-\eqref{DiscIBP}. Define the errors $\uEb^{n}=\ucb^{n}-\ueb^{n}$ and $\XEb^{n}=\Xcb^{n}-\Xeb^{n}$, where $\ueb^{n}$, $\Xeb^{n}$ are simply the quantities $\ueb(\x,n\dt)$,$\Xeb(n\dt)$ respectively for $\x\in \Grid$. Suppose we choose  $\dx^{2} \propto \dt$. Then we have\footnote{When applied to gridfunctions, $\|\cdot\|$ refers to the averaged 2 norm, i.e. $\|\uEb\| = \frac{1}{\length^2}\left(\sum_{\x\in \Grid}\sum_{\alpha=1}^{2}\uE^2_{\alpha}(\x) \dx^2\right)^{1/2}$. When applied to vectors like $\XEb$, $\|\cdot\|$ is simply the Euclidean distance}
\begin{equation}
\max_{0\leq n \leq \frac{T}{\dt}} \|\uEb^{n}\| = O(\dt) \quad \text{and} \quad \max_{0\leq n \leq \frac{T}{\dt}} \|\XEb^{n}\| = O(\dt)
\end{equation}
\end{theorem}

This is equivalent to the definition of convergence used in \cite{lax2005survey}. As in Lax and Richtmyer's paper, proving convergence will involve establishing both consistency and stability of the scheme. For the consistency portion, we define the residues $\ures^{n},\pres^{n}$ and $\Xres^{n}$:

\begin{align}
\ures^{n}&=\rho\left(\frac{\ueb^{n+1}-\ueb^{n}}{\dt} + (\ueb^{n}\cdot\D)\ueb^{n}\right) + \D \pe^{n+1} \label{ResMom}\\
           &\qquad -\mu(\Dp\cdot\Dm)\ueb^{n+1} - \F(\Xeb^{n})\dc(\x-\Xeb^{n}) \notag\\
\pres^{n+1}  &= \D \cdot \ueb^{n+1} \label{ResIncomp}\\
\Xres^{n}  &= \frac{\Xeb^{n+1} - \Xeb^{n}}{\dt} - \sum_{\x\in \Grid}\ueb^{n}(\x)\dc(\x-\Xeb^{n})\dx^2 \label{ResIBP}
\end{align}

and claim they converge uniformly to 0.

\begin{lemma}[Consistency] \label{Consistency}
The residues $\ures^{n},\pres^{n}$ and $\Xres^{n}$ converge to $0$ uniformly for all $n$. More specifically, we find
\begin{align}
    \|\ures^{n}\| &= O(\dt) & \|\pres^{n}\| &= O(h^2 = \dt) \\
    \|D^{-}_{\alpha}\pres^{n}\| &= O(h^2 = \dt) & \|\Xres\| &= O(\dt)
\end{align}
\end{lemma}

The proof of this lemma involves repeated application of Taylor's theorem to equations \eqref{ResMom}-\eqref{ResIBP}; see \ref{sec:lem} for the details. The remainder of this paper constructs a stability bound of the form

\begin{align}
\rm{error}(n) &\leq \frac{1}{\Lambda}\left(e^{T\Lambda}-1\right)\max_{0\leq i\leq T/\dt}\rm{residues(i)}\label{StabBoundForm}
\end{align}

For some constant $\Lambda$. When combined with consistency, this inequality provides a uniform bound on the error. Before delving into the proof of the theorem proper, we first make some remarks about the error. By subtracting equations \eqref{ResMom}-\eqref{ResIBP} from \eqref{DiscMom}-\eqref{DiscIBP}, we obtain an iterative relation for the error

\begin{align}
\begin{split}
\frac{\uEb^{n+1}-\uEb^{n}}{\dt} &= - (\ucb^{n}\cdot\D)\ucb^{n} + (\ueb^{n}\cdot\D)\ueb^{n} -\frac{1}{\rho}\D \pE^{n+1} + \frac{\mu}{\rho}(\Dp\cdot\Dm)\uEb^{n+1} \\
&\qquad \qquad +- \frac{1}{\rho}\ures^{n} +\frac{1}{\rho}\F(\Xcb^{n})\dc(\x-\Xcb^{n}) - \frac{1}{\rho}\F(\Xeb^{n})\dc(\x-\Xeb^{n}) 
\end{split} \label{ErrMom}\\
 \D\cdot\uEb^{n+1} &= -\pres^{n+1} \label{ErrIncomp} \\
 \frac{\XEb^{n+1}-\XEb^{n}}{\dt} &= \sum_{\x\in \Grid} \ucb^{n}(\x)\dc(\x-\Xcb^{n})\dx^2 - \ueb^{n}(\x)\dc(\x-\Xeb^{n})\dx^2 - \Xres^{n} \label{ErrIBP}
\end{align}

Where $\pE^{n} = \pc^{n} - \pe^{n}$. By adding and subtracting the terms

\begin{equation*}
(\ucb^{n} \cdot \D)\ueb^{n} \quad\text{and} \quad \F(\Xeb^{n})\dc(\x-\Xcb^{n})
\end{equation*}

we may rewrite \eqref{ErrMom} as

\begin{equation}
\begin{split}
\frac{\uEb^{n+1}-\uEb^{n}}{\dt} &= -(\ucb^{n}\cdot\D)\uEb^{n} - (\uEb^{n}\cdot\D)\ueb^{n} - \frac{1}{\rho}\D q^{n+1} + \frac{\mu}{\rho}(\Dp\cdot\Dm)\uEb^{n+1} \\
&\qquad - \frac{1}{\rho}\ures^{n}  + \frac{1}{\rho}\F(\Xeb^{n})(\dc(\x-\Xcb^{n})-\dc(\x-\Xeb^{n})) \\
&\qquad + \frac{1}{\rho}(\F(\Xcb^{n})-\F(\Xeb^{n}))\dc(\x-\Xcb^{n})\label{ErrMom1}
\end{split}
\end{equation}

Similarly, we may recast \eqref{ErrIBP} as

\begin{equation}
\frac{\XEb^{n+1}-\XEb^{n}}{\dt} = \sum_{\x\in \Grid} \uEb^{n}(\x)\dc(\x-\Xcb^{n})h^2 + \ueb^{n}(\x)(\dc(\x-\Xcb^{n}) - \dc(\x-\Xeb^{n}))h^2 - \Xres^{n} \label{ErrIBP1}
\end{equation}

We can further process these equations by decomposing the fluid error $\uEb$ into discretely divergence-free and discretely curl-free components:

\begin{lemma}[Discrete Hemholtz Decomposition] \label{Decomp}
$\uEb^{n}$ may be decomposed into the sum
\begin{equation}
	\uEb^{n} = \usEb^{n} + \D\psi^{n}
\end{equation}
Where 
\begin{align}
\D\cdot\usEb^{n} = 0 \\
\upE^{n} \perp \ker(\D)
\end{align}
Where the latter condition ensures $\upE$ is unique. As an immediate consequence, we have $\D\upE^{n} \perp \ker(\D\cdot)$
\end{lemma}

To prove this, it is sufficient to demonstrate that $-\D$ and $\D\cdot$ are adjoint. The proof can be found in \ref{sec:lem}. Applied to the linear terms in equations \eqref{ErrMom1} and \eqref{ResIncomp}, the above lemma gives us

\begin{equation}
\begin{split}
\frac{\usEb^{n+1} - \usEb^{n}}{\dt} &+ \frac{\D\psi^{n+1}-\D\psi^{n}}{\dt} = -(\ucb^{n}\cdot\D)\uEb^{n} - (\uEb^{n}\cdot\D)\ueb^{n} \\
&\qquad  - \frac{1}{\rho}\D \pE^{n+1}  - \frac{1}{\rho}\ures^{n} + \frac{1}{\rho}(\F(\Xcb^{n})-\F(\Xeb^{n}))\dc(\x-\Xcb^{n}) \\
&\qquad + \frac{1}{\rho}\F(\Xeb^{n})(\dc(\x-\Xcb^{n})-\dc(\x-\Xeb^{n})) \\
&\qquad + \frac{\mu}{\rho}(\Dp\cdot\Dm)(\usEb^{n+1} + \D\upE^{n+1})
\end{split} \label{ErrMom2}
\end{equation}
\begin{equation}
     \pres^{n+1} = \D\cdot\D \upE^{n+1}\label{ErrIncomp1}
\end{equation}

The last observation we make is that the component $\D\upE^{n}$---i.e. the deviation in the error from the incompressibility constraint subspace---can be bounded by the residue $\pres^{n}$. This is significant, as it ensures the error also approximately satisfies the discretely divergence-free condition. This will simplify our problem, as we will only need to prove a stability bound on the discrete divergence-free field $\usEb^{n}$, and the bound for $\uEb^{n}$ will follow. We summarize this claim in the lemma below.

\begin{lemma}[Stability bound on deviation from constraint subspace] \label{ConstraintBound}
Let $\upE^{n+1}$ satisfy \eqref{ErrIncomp1}. The gridfunction $\D\upE^{n}$ satisfies the bound
\begin{equation}
\|\D\upE^{n}\|^2 \leq \frac{1}{\lambda}\|\pres^{n}\|^{2}
\end{equation}
Where $\lambda$ is a lower bound, uniform in $\dx$, to the non-zero eigenvalues of $-\D\cdot\D$.
\end{lemma}

The crux of this lemma is demonstrated by analysing the spectrum of the discrete Laplacian $\D\cdot\D$. Since the boundary conditions are periodic, the matrix is circulant and can be diagonalised by the discrete Fourier transform. Having established the necessary context we move on to a proof of the main theorem

\section{Proof of the theorem}
\label{sec:proof}
The proof of the theorem will call upon an additional lemma, which is an iterative inequality that we use inductively to establish a stability bound in the form of \eqref{StabBoundForm}.

\begin{lemma}[Iterative inequality]\label{IterInequalities}
Assume $\usEb^{n}$ and $\XEb^{n}$ satisfy equations \eqref{ErrIBP1} and \eqref{ErrMom2}. Let $T_{0}$ be an arbitrary constant measured in units of time and suppose $|\uc^{n}_{\alpha}(\x)| \leq \Muc$. Define
\begin{align}
    A^{n} &:= \frac{1}{2}\|\usEb^{n}\|^2 + \frac{\dt\mu}{\rho}\sum_{\alpha=1}^{2}\|D^{-}_{\alpha}\usEb^{n}\|^2 \\
    B^{n}&:= \frac{1}{T_{0}^2}\|\XEb^{n}\|^2 \\
    C^{n}&:= \frac{1}{2T_{0}\lambda}\|\pres^n\|^2 + \frac{\mu}{\rho \lambda}\sum_{\alpha=1}^{2}\|D^{-}_{\alpha}\pres^{n}\|^2 + \frac{T_{0}}{2\rho^2}\|\ures^{n}\|^2 \\
    D^{n}&:=\frac{1}{T_{0}}\left(\frac{1}{\lambda}\|\pres^{n}\|^2 + \|\Xres^{n}\|^2 \right)
\end{align}
then
\begin{equation}
\left[
\begin{array}{c}
A^{n+1} \\
B^{n+1}
\end{array}
\right]
\leq
(\mathbf{I} + \LexpM^{n} \dt)
\left[
\begin{array}{c}
A^{n} \\
B^{n}
\end{array}
\right]
+
\dt
\left[
\begin{array}{c}
C^{n}\left(1+\dt k_{1}^{n}\right) \\
D^{n}(1 + \dt k_{2})
\end{array}
\right] \label{IterIneq}
\end{equation}
where the inequality holds for both lines of \eqref{IterIneq}, and $\LexpM^{n}$ is the $2\times2$  matrix
\begin{equation*}
    \left(
    \begin{array}{cc}
        \frac{1}{T_{0}}+(1+\frac{\dt}{T_{0}})k_{1}^{n}  & 
        \frac{1}{T_{0}}\\
        \frac{2}{T_{0}} & \frac{1}{T_{0}}+(1+\frac{\dt}{T_{0}})k_{2}
    \end{array}
    \right)
\end{equation*}
and
\begin{align*}
    k_{1}^{n} &= \left(4T_{0} \MDu^2 + \frac{\rho}{\mu}\Muc^2 + \frac{\LF \Md + \MF \Ld}{2\rho^2}T_{0}^{3} + \frac{1}{T_{0}} \right) \\
    k_{2} &= \left(2\Mu^2 L_{\delta_{c}}^2 L^4 T_{0} + L^4\frac{1}{T_{0}} \Md^2 + \frac{1}{T_{0}}\right)
\end{align*}
\end{lemma}
The constant $T_{0}$ is introduced to maintain dimensional consistency throughout the proof, and plays no significant role in the veracity of the theorem. The result of this lemma is obtained through repeated applications of Cauchy-Schwarz and other inequalities to the iterative error equations \eqref{ErrIBP1}, \eqref{ErrMom2}, and \eqref{ErrIncomp1} (see section \ref{sec:lem} for the demonstration). It is important to note also that $k_{1}^{n}$ and by extension $\LexpM^{n}$ depend implicitly on $n$ through the presence of the bound $\Muc$. In order to obtain a bound on the growth of the error, we must find a bound for the computed solution. On the other hand, we can bound the growth of the computed solution if we have a bound on the error (since the error tells us the deviation of the computed solution from the bounded exact solution). This may initially seem to be a circular argument, however, a second glance at \eqref{IterIneq} tells that this is not the case: we only require a bound on $\ucb^{n}$, to obtain a bound on $\usEb^{n+1}$---which may in turn be used to control $\uEb^{n+1}$. Furthermore, since $\uEb^{n+1}$ tells us how far $\ucb^{n+1}$ is from $\ueb^{n+1}$, and since $\ueb$ is---by assumption---uniformly bounded, we can choose $\dt$ small enough to maintain a bound on $\ucb^{n+1}$. This allows us to form an inductive argument which proves both convergence and boundedness of $\ucb^{n}$ simultaneously. We begin with a corollary of the above lemma

\begin{corollary}\label{Cor:Iter}
Define $\Lexp^{n}$ to be the 1-norm of $\LexpM^{n}$
\begin{equation*}
\Lexp^{n}:= \max\{\Lexp_{11}^{n} + \Lexp_{21}^{n},\Lexp_{12}^{n} + \Lambda_{22}^{n}\}
\end{equation*}
then
\begin{equation}
    A^{n+1} + B^{n+1} \leq (1+ \dt\Lexp^{n})(A^{n}+B^{n}) + \dt E^{n}\label{IterIneq1}
\end{equation}
where 
\begin{equation}
E^{n} =C^{n}\left(1+\dt k_{1}^{n}\right) + D^{n}(1 + \dt k_{2})
\end{equation}
\end{corollary}

\begin{proof}[Proof of the convergence theorem] For clarity, quantities in this proof indexed by $n$ or $m$ will be made parenthetic to distinguish from exponentiation. We begin by noting that there exist upper bounds $\Lexp$ and $k_{1}$ such that 
\begin{equation}
    |\uc_{\alpha}^{(n)}(\x)| \leq 2\Mu \quad\Rightarrow \quad\Lexp^{(n)} \leq \Lexp \quad\text{and}\quad k_{1}^{(n)} \leq k_{1}
\end{equation}
For all $n$. Fix $\dt$ and let $\mathcal{P}(n)$ be the proposition that
\begin{equation}
A^{(n)} + B^{(n)} \leq \max_{m\leq n}\left(E^{(m)}\right)\frac{(1+\dt\Lexp)^{n} - 1}{\Lambda} \label{hyp1}
\end{equation}
and
\begin{equation}
    |\uc_{\alpha}^{(n)}(\x)| \leq 2\Mu \label{hyp2}
\end{equation}
Clearly $\mathcal{P}(0)$ is true, since $\uc_{\alpha}^{(0)}(\x)=\ue_{\alpha}^{(0)}(\x)$ and $\Xcb^{(0)}=\Xeb^{(0)}$, so $A^{(0)},B^{(0)}=0$. Now suppose $\mathcal{P}(n)$ is true. Then from corollary \ref{Cor:Iter} and the inductive hypotheses \eqref{hyp1} and \eqref{hyp2}

\begin{align*}
    A^{(n+1)} + B^{(n+1)} &\leq (1+ \dt\Lexp^{(n)})(A^{(n)}+B^{(n)}) + \dt E^{(n)} \\
    &\leq (1+ \dt\Lexp)\left(\max_{m\leq n}(E^{(m)})\frac{(1+\dt\Lexp)^{n} - 1}{\Lexp}\right) + \dt E^{(n)} \\
    &= \max_{m\leq n}(E^{(m)})\frac{(1+\dt\Lexp)^{n+1} - 1}{\Lexp} - \dt \max_{m\leq n}(E^{(m)}) + \dt E^{(n)} \\
    &\leq \max_{m\leq n+1}(E^{(m)})\frac{(1+\dt\Lexp)^{n+1} - 1}{\Lexp}
\end{align*}

This completes the first half of the inductive step. Now we must show that $|\uc_{\alpha}^{(n+1)}(\x)| \leq 2\Mu$. Since we're only interested in times earlier than $T$, we may assume $(n+1)\dt \leq T$, so that

\begin{align*}
    \|\usEb^{(n+1)}(\x)\| \leq A^{(n+1)} &\leq \max_{m\leq n+1}(E^{(m)})\frac{(1+\dt\Lexp)^{n+1} - 1}{\Lexp} \\
    &\leq \max_{m\leq n+1}(E^{(m)})\left(e^{\Lexp T}-1\right)
\end{align*}

From lemmas \ref{Decomp} and \ref{ConstraintBound}

\begin{align*}
    \|\uEb^{(n+1)}\|^2 &= \|\usEb^{(n+1)}\|^2 + \|\D\psi^{(n+1)}\|^2 \\
    & \leq \|\usEb^{(n+1)}\|^2 + \frac{1}{\lambda}\|\pres^{(n+1)}\|^2 \\
    &\leq \max_{m\leq n+1}(E^{(m)})\left(e^{\Lexp T}-1\right) + \frac{1}{\lambda}\|\pres^{(n+1)}\|^2
\end{align*}

We can use to establish a pointwise bound on $\uE_{\alpha}^{(n+1)}(\x)$

\begin{align*}
    \left(\uE_{\alpha}^{(n+1)}\right)^{2} &\leq \sum_{\alpha=1}^{2}\sum_{\x\in\Grid} \left(\uE_{\alpha}^{(n+1)}\right)^{2} \\
    &= \|\uEb^{(n+1)}\|^2 \frac{\length^2}{\dx^2}
\end{align*}

To conclude,

\begin{align*}
    \left|\uc^{(n+1)}_{\alpha}(\x)\right| &= \left|\uc^{(n+1)}_{\alpha}(\x)-\ue^{(n+1)}_{\alpha}(\x)+\ue^{(n+1)}_{\alpha}(\x)\right| \\
    &\leq \left|\uE^{(n+1)}_{\alpha}(\x)|+|\ue^{(n+1)}_{\alpha}(\x)\right| \\
    &\leq \frac{L}{\dx}\sqrt{\max_{m\leq n+1}\left(E^{(m)}\right)\left(e^{\Lexp T}-1\right) + \frac{1}{\lambda}\|\pres^{(n+1)}\|^2} + \Mu
\end{align*}

As a consequence of lemma \ref{Consistency} the term $E^{(n)} = O(\dt^2)$ and $\|\pres^{n}\|=O(\dt)$ {\it independently} of $n$ (assuming as always that $\dt \propto \dx^2$). Therefore, if $\dt$ happens to be chosen sufficiently small,

\begin{equation}
    \frac{L}{\dx}\sqrt{\max_{m\leq n+1}\left(E^{(m)}\right)\left(e^{\Lexp T}-1\right) + \frac{1}{\lambda}\|\pres^{(n+1)}\|^2}
\end{equation}

will be less than $\Mu$, and the second half of the inductive step will be proven. This closes the loop and proves $\mathcal{P}(n)$ true for all $n$ such that $n \dt \leq T$, provided $\dt$ is sufficiently small. From reasoning similar to the above argument, it follows from \eqref{hyp1} that

\begin{align*}
    \|\uEb^{(n)}\|^2 &= \|\usEb^{(n)}\|^2 + \|\D\psi^{(n)}\|^2 \\
    & \leq A^{(n)} + \frac{1}{\lambda}\|\pres^{(n)}\|^2 \\
    &\leq \max_{m\leq n}\left(E^{(m)}\right)\left(e^{\Lexp T}-1\right) + \frac{1}{\lambda}\|\pres^{(n)}\|^2
\end{align*}
and
\begin{align*}
    \|\XEb^{(n)}\|^2 & \leq B^{(n)} \\
    &\leq \max_{m\leq n}\left(E^{(m)}\right)\left(e^{\Lexp T}-1\right)
\end{align*}
Both error bounds are $O(\dt^2)$, proving the convergence result.
\end{proof}

\section{Proofs of the supporting lemmas} \label{sec:lem}

Now we go back and prove our supporting lemmas

\begin{proof}[Proof of Lemma \ref{Consistency}]
    As mentioned above, the proof is mostly a matter of evaluating each of the terms via Taylor expansions. Let $C^{n}$ denote the set of periodic functions from $\Omega$ to $\mathbb{R}$ for which all $n$th order partial derivatives are continuous. First, we note that when $\psi$ is either $C^{3}$ or $C^{4}$, we have the relation
    \begin{equation}
    D^{0}_{\alpha}\psi = \frac{\psi(\x_{ij}+h\e_{\alpha}) - \psi(\x_{ij}-h\e_{\alpha})}{2h} = \left.\frac{\partial \psi}{\partial x_{\alpha}}\right|_{\x_{ij}} + \left.\frac{h^2}{6} \frac{\partial^3 \psi}{\partial x_{\alpha}^3}\right|_{\x_{ij}} + o(h^2)
    \end{equation}
    where $\e_{1},\e_{2}$ are standard basis vectors. We have
    \begin{align*}
    (\ueb^{n}\cdot\D)u_{\alpha}^{n} &= \sum_{\beta=1}^{2}u^{n}_{\beta}D_{\beta}^{0}u_{\alpha}^{n} \\
    &= \sum_{\beta=1}^{2}u^{n}_{\beta}\left(\left.\frac{\partial u_{\alpha}}{\partial x_{\beta}}\right|_{\x_{ij}} + \frac{h^2}{6} \left.\frac{\partial^3 u_{\alpha}}{\partial x_{\beta}^3}\right|_{\x_{ij}} + o(h^2)\right) = (\ueb^{n}\cdot\nabla)u_{\alpha}^n + O(h^2)
    \end{align*}
    for the pressure term, we get
    \begin{align*}
    D^{0}_{\alpha}\pe^{n+1} &= D^{0}_{\alpha}\pe^{n} + \dt D^{0}_{\alpha}\left.\frac{\partial \pe}{\partial t}\right|^{n} + o(\dt) \\
    &= \left.\frac{\partial \pe}{\partial x_{\alpha}}\right|^{n}_{\x_{ij}} + \frac{h^2}{6} \left.\frac{\partial^3 \pe}{\partial x_{\alpha}^3}\right|^{n}_{\x_{ij}} + o(h^2) \\ 
    & \qquad + \dt\left(\left.\frac{\partial^2 \pe}{\partial x_{\alpha}\partial t}\right|^{n}_{\x_{ij}} + \frac{h^2}{6}\left.\frac{\partial^4 \pe}{\partial x_{\alpha}^{3}\partial t}\right|^{n}_{\x_{ij}} + o(h^2)\right) + o(\dt) \\
    &= \left.\frac{\partial \pe}{\partial x_{\alpha}}\right|^{n}_{\x_{ij}} + O(\dt) + O(h^2)
    \end{align*}
    The Taylor expansion of the discretised Laplacian $\Dp\cdot\Dm$ is
    \begin{equation}
    \begin{split}
    \Dp\cdot\Dm\psi &= \sum_{\alpha=1}^{2}\frac{\psi(\x_{ij}+h\e_{\alpha}) - 2\psi(\x_{ij}) + \psi(\x_{ij}-h\e_{\alpha})}{h^2} \\
    &=\left\{
    \begin{array}{lr}
    \nabla^2 \psi + \sum_{\alpha=1}^{2}\frac{h^2}{12}\frac{\partial^4 \psi}{\partial x_{\alpha}^4} + o(h^3) & \text{if } \psi \text{ is } C^{4}\\
    \nabla^2 \psi + o(h) & \text{if } \psi \text{ is } C^{3}
    \end{array}
    \right.
    \end{split}
    \end{equation}
    So the viscosity term works out to be
    \begin{align*}
    \Dp\cdot\Dm u_{\alpha}^{n+1} &= \Dp\cdot\Dm u_{\alpha}^{n} + \dt\Dp\cdot\Dm\left.\frac{\partial u_{\alpha}}{\partial t}\right|^{n} + o(\dt) \\
     &= \nabla^2 u_{\alpha}^{n} + h^2\sum_{\alpha=1}^{2}\frac{1}{12}\left.\frac{\partial^4 u_{\alpha}}{\partial x_{\alpha}^4}\right|^{n}_{\x_{ij}}+ \dt \left(\left.\frac{\partial}{\partial t} \nabla^2 u_{\alpha}\right|^{n}_{\x_{ij}} + o(h)\right) + o(\dt) \\
     &= \nabla^2 u_{\alpha}^{n} + O(\dt) + O(h^2)
    \end{align*}
    Finally, we have
    \begin{equation*}
    \frac{\ueb^{n+1}-\ueb^{n}}{\dt} = \left.\frac{\partial \ueb}{\partial t}\right|^{n}_{\x_{ij}} + O(\dt)
    \end{equation*}
    Combining the above expansions, we obtain an expression for $\ures^{n}$
    \begin{align*}
        \ures^{n} &= \rho\left(\frac{\partial \ueb}{\partial t} + (\ueb^{n}\cdot\nabla)\ueb^{n} \right) + \nabla\pe^{n} - \mu\nabla^2\ueb^{n} - \F(\Xeb^{n})\dc(\x-\Xeb^{n}) + O(\dt) + O(h^2)\\
        &= O(\dt) + O(h^2)
    \end{align*}
    As $h^2 \propto \dt$, $\|\ures^{n}\|= O(\dt)$ as desired. For the continuity equation, we have
    \begin{equation*}
    \pres^{n+1} = \D\cdot\ueb^{n+1} = \sum_{\alpha=1}^{2}D^{0}_{\alpha}u_{\alpha}^{n+1} = \sum_{\alpha=1}^{2}\left.\frac{\partial u_{\alpha}}{\partial x_{\alpha}}\right|^{n+1}_{\x_{ij}} + O(h^2) = O(h^2)
    \end{equation*}
    We also have
    \begin{align*}
    D^{-}_{\alpha}\pres^{n+1} &= D^{-}_{\alpha}\D\cdot\ueb^{n+1} \\
    &= D^{-}_{\alpha}\left(\sum_{\beta=1}^{2}\left.\frac{\partial u_{\beta}}{\partial x_{\beta}}\right|^{n+1}_{\x_{ij}} + \frac{h^2}{6}\left.\frac{\partial^3 u_{\beta}}{\partial x_{\beta}^3}\right|^{n+1}_{\x_{ij}} + o(h^2)\right) \\
    &= \sum_{\beta=1}^{2}\frac{h^2}{6}\left.\frac{\partial^4 u_{\beta}}{\partial x_{\beta}^3\partial x_{\alpha}}\right|^{n+1}_{\x_{ij}} + o(h^2) = O(h^2)
    \end{align*}
    To prove consistency of the residue $\Xres^{n}$, we need to estimate the error between the integral 
    \begin{equation*}
    \int u_{\alpha}(\x,t)\dc(\x-\Xeb)d\x \label{ContAdvec}
    \end{equation*}
    and the discrete sum 
    \begin{equation*}
    \sum_{\x\in \Grid}u_{\alpha}^{n}(\x_{ij})\dc(\x_{ij}-\Xeb^{n})h^2
    \end{equation*}
    Consider a cell of width $h$, centred at $\x$: Since $\ueb\dc(\x-\Xeb)$ is $C^2$, we have by Taylor's theorem
    \begin{equation*}
        \int_{-\frac{h}{2}}^{\frac{h}{2}}\int_{-\frac{h}{2}}^{\frac{h}{2}} u_{\alpha}(\x + \mathbf{y},t)\dc(\x + \mathbf{y}-\Xeb) d\mathbf{y} = \ueb(\x)\dc(\x-\Xeb)h^2 + O(h^4)
    \end{equation*}
        Where the linear term vanishes by symmetry. Summing over all $\x$, we obtain
    \begin{align*}
    \sum_{\x\in \Grid}u_{\alpha}^{n}(\x)\dc(\x-\Xeb^{n})h^2 = \int_{0}^{L}\int_{0}^{L}u_{\alpha}\delta_{c}(\x-\Xeb)dx_{1}dx_{2} + O(h^2)
    \end{align*}
 Finally, as
    \begin{equation*}
        \frac{\tilde{X}_{\alpha}^{n+1} - \tilde{X}_{\alpha}^{n}}{\dt} = \left.\frac{\partial \tilde{X}_{\alpha}}{\partial t}\right|^{n} + \dt\left.\frac{\partial^2 \tilde{X}_{\alpha}}{\partial t^2}\right|^{n} + o(\dt)
    \end{equation*}
    We conclude that $\Xres^n = O(\dt) + O(h^2)$
\end{proof}

\begin{proof}[Proof of Lemma \ref{Decomp}]
Let $V$ be the space of vector valued gridfunctions; let $S$ be the space of scalar valued grid functions; and let $\{\e_{1},\e_{2}\}$ denote the standard basis vectors. Then for any $\uEb\in V$ and $\psi\in S$, we can use periodicity of $\D$ and summation by parts to show

\begin{align*}
(\psi,\D \cdot\uEb) &= \frac{h^2}{L^2}\sum_{\alpha=1}^{2}\sum_{\x\in \Grid}\frac{v_{\alpha}(\x_{ij}+h\e_{\alpha})-v_{\alpha}(\x_{ij}-h\e_{\alpha})}{2h}\psi(\x_{ij}) \\
&= \frac{h^2}{L^2}\sum_{\alpha=1}^{2}\frac{1}{2h}\sum_{\x'\in \Grid}(v_{\alpha}(\x'_{ij})\psi(\x'_{ij}-h\e_{\alpha})-\sum_{\x''\in \Grid}v_{\alpha}(\x''_{ij})\psi(\x''_{ij}+h\e_{\alpha}) \\
&= -\frac{h^2}{L^2}\sum_{\alpha=1}^{2}\sum_{\x\in \Grid}v_{\alpha}(\x_{ij})\frac{\psi(\x_{ij}+h\e_{\alpha})-\psi(\x_{ij}-h\e_{\alpha})}{2h} \\
&= -(\D\psi,\uEb)
\end{align*}

So the discrete divergence operator $\D\cdot$ is minus the adjoint of the discrete gradient operator $\D$, and it follows from the theory of dual spaces\footnote{\cite{lax2014linear} is a good reference} that the kernel of $\D\cdot$ is the orthogonal complement of the range of $\D$. Thus any $\uEb\in V$ may be uniquely decomposed into orthogonal components

\begin{equation}
\uEb = \usEb + \D \psi
\end{equation}

for some $\usEb \in \mathrm{ker}(\D)$ and $\psi\in S$. Since $\D$ has a non-trivial kernel, $\psi$ is not uniquely defined until we specify the further stipulation that $\psi$ should be orthogonal to $\mathrm{ker}(\D)$
\end{proof}

\begin{proof}[Proof of Lemma \ref{ConstraintBound}]
Consider equation \eqref{ErrIncomp1}. Taking the inner product with $\psi^{n}$ and applying Cauchy-Schwarz, we have

\begin{equation}
\|\D\psi^{n}\|^2 = -(\psi^{n},\D\cdot\D\psi^{n}) = (\psi^{n},\pres^{n})\leq \|\psi^{n}\|\|\pres^{n}\|
\end{equation}

What remains is to bound $\psi^{n}$ by $-\D\cdot\D\psi^{n} = \pres^{n}$. Although $-(\D\cdot\D)$ is not strictly positive definite, $\psi^{n}$ lies orthogonal to its kernel. This means that the Raleigh quotient 

\begin{equation*}
-\frac{(\psi^{n},\D\cdot\D\psi^{n})}{(\psi^{n},\psi^{n})}
\end{equation*}

will be strictly positive. To show this, we note that $-\D\cdot\D$ is circulant, and can be expanded by a basis of Fourier modes. Let

\begin{equation*}
\psi_{rs}(\x_{ij}) = \omega^{ir+js} \qquad \text{ where $\omega$ is a primitive $N$-th root of unity}
\end{equation*}

Then $\{\psi_{rs}:0\leq r,s < N\}$ are orthonormal eigen-gridfunctions of $-\D\cdot\D$, and a basis of $S$. The corresponding eigenvalues are

\begin{equation}
\lambda_{rs} = \frac{1}{4h^2}\left[\left(\sin\left(\frac{2\pi}{N}r\right)\right)^2 + \left(\sin\left(\frac{2\pi}{N}s\right)\right)^2\right]
\label{LapSpec}
\end{equation}

Suppose for simplicity that $N$ is even. Then the kernel of $-\D\cdot\D$ is 4 dimensional, spanned by $\{\psi_{rs}:r,s\in\{0,N/2\}\}$ (this is the subspace of gridfunctions which are either constant or alternating in sign along each dimension). We note in particular, that this sub-space is equal to the kernel of $\D$. Since $\psi^{n}$ is defined orthogonal to the kernel of $\D$, we can find an orthogonal expansion of the form

\begin{equation*}
\psi^{n} = \sum_{\substack{0\leq r,s<N \\ r,s\notin\{0,N/2\}}} a_{rs}\psi_{rs}
\end{equation*}

From which we obtain

\begin{equation*}
\begin{split}
    -(\psi^{n},\D\cdot\D\psi^{n}) = \frac{1}{h^2} \sum_{\substack{0\leq r,s<N \\ r,s \notin \{0,N/2\}}}\lambda_{rs}a_{rs}^{2} &\geq \min_{r,s}\{\lambda_{rs}\}\frac{1}{h^2} \sum_{\substack{0\leq r,s<N \\ r,s \notin \{0,N/2\}}} a_{rs}^{2} \\
    &= \min_{r,s}\{\lambda_{rs}\}(\psi^{n},\psi^{n})
\end{split}
\end{equation*}

Defining 
$\lambda=\min\{\lambda_{rs}:0\leq r,s <N,h>0\}$, we obtain the bound $\lambda\|\psi^{n}\|^{2} \leq (\psi^{n},\pres^{n}) \leq \|\psi^{n}\|\|\pres^{n}\|$. Whence we find, as desired,

\begin{equation}
\|\D\psi^{n}\|^2 \leq \|\psi^{n}\|\|\pres^{n}\| \leq \frac{1}{\lambda}\|\pres^{n}\|^2
\end{equation}
\end{proof}

\begin{proof}[Proof of Lemma \ref{IterInequalities}]
 We begin with the iterative inequality for $A^{n}$. The objective here is to use \eqref{ErrMom2} to calculate the growth of the discretely divergence-free error $\usEb^{n}$ in one timestep. The first step to do this is to extract and isolate all instances of $\usEb^{n+1}$ from terms evaluated at $n\dt$. Assume $|\uc_{\alpha}^{n}(\x)| \leq \Muc$. Taking the inner product of equation \eqref{ErrMom2} with $\usEb^{n+1}$, the discretely curl-free terms vanish, leaving us with

\begin{equation}
\begin{split}
\|\usEb^{n+1}\|^2 = (\usEb^{n+1},\usEb^{n}) - \dt(\usEb^{n+1},(\ucb^{n}\cdot\D)\uEb^{n}) - \dt(\usEb^{n+1},(\uEb^{n}\cdot\D)\ueb^{n}) \\
- \dt\frac{\mu}{\rho}\sum_{\alpha=1}^{2}\|D_{\alpha}^{-}\usEb^{n+1}\|^{2} + \frac{\dt}{\rho}(\usEb^{n+1},(\F(\Xcb^{n})-\F(\Xeb^{n}))\dc(\x-\Xcb^{n})) \\
+ \frac{\dt}{\rho}(\usEb^{n+1},\F(\Xeb^{n})(\dc(\x-\Xcb^{n})-\dc(\x-\Xeb^{n}))) - \frac{\dt}{\rho}(\usEb^{n+1},\ures^{n})
\end{split} 
\end{equation}

Where we have used the fact that $\D,\Dp,\Dm$ commute, and the fact that $\Dm$ is minus the adjoint of $\Dp$. Applying Cauchy-Schwarz to each inner product, we obtain

\begin{equation}
\begin{split}
\|\usEb^{n+1}\|^2 + &\dt\frac{\mu}{\rho}\sum_{\alpha=1}^{2}\|D^{-}_{\alpha}\usEb^{n+1}\|^{2} \leq \\ &\|\usEb^{n+1}\|\Bigg(\|\usEb^{n}\| + \sqrt{2}\dt \Muc \left(\sum_{\beta=1}^{2}\|D^{-}_{\beta}\uEb^{n}\|^{2} \right)^{1/2}  \\
&\qquad + \dt \MDu \|\uEb^{n}\| + \frac{\dt}{\rho}\|(\F(\Xcb^{n})-\F(\Xeb^{n}))\dc(\x-\Xcb^{n})\| \\
&\qquad+ \frac{\dt}{\rho}\|\F(\Xeb^{n})(\dc(\x-\Xcb^{n})-\dc(\x-\Xeb^{n}))\| - \frac{\dt}{\rho}\|\ures^{n}\| \Bigg)
\end{split}
\end{equation}

Where we've treated the non-linear advection terms specially as follows:

\begin{align*}
\|(\ucb^{n}\cdot\D)\uEb^{n})\|^2 &= \sum_{\x\in \Grid}\sum_{\alpha=1}^{2}\left( \sum_{\beta = 1}^{2}  \uc^{n}_{\beta}(\x)D_{\beta}^{0}v^{n}_{\alpha}(\x)\right)^{2} \frac{h^2}{L^2} \\
&\sum_{\x\in \Grid}\sum_{\alpha=1}^{2}\sum_{\beta'=1}^{2} \left(\uc^{n}_{\beta'}\right)^{2} \sum_{\beta = 1}^{2}  \left(D_{\beta}^{0}v^{n}_{\alpha}\right)^{2} \frac{h^2}{L^2} \\
&\leq 2 \Muc^2 \sum_{\x\in \Grid} \sum_{\alpha= 1}^{2}\sum_{\beta= 1}^{2}\left(D_{\beta}^{0}v^{n}_{\alpha}\right)^{2}  \frac{h^2}{L^2} \\
&\leq2 \Muc^2 \sum_{\beta=1}^{2} \|D_{\beta}^{0}v^{n}_{\alpha}\|^{2}   \\
&\leq 2 \Muc^2 \sum_{\beta=1}^{2} \|D_{\beta}^{-}v^{n}_{\alpha}\|^{2}
\end{align*}

(In the last line we use the fact that $\|D^{-}_{\beta}\uEb\| \geq \|D_{\beta}^{0}\uEb\|$. This is true because of the relation $D_{\beta}^{0} = \frac{1}{2}(D^{+}_{\beta} + D^{-}_{\beta})$ and $\|D^{+}_{\beta}\uEb\| = \|D^{-}_{\beta}\uEb\|$, which is a consequence of periodicity. We also have

\begin{align*}
\|(\uEb^{n}\cdot\D)\ueb^{n}\|^{2} &= \sum_{\x\in \Grid}\sum_{\alpha=1}^{2}\sum_{\beta = 1}^{2} \left( \uE^{n}_{\beta}(\x)D_{\beta}^{0}\ue^{n}_{\alpha}(\x) \right)^2 \frac{h^2}{L^2} \\
&\leq \sum_{\x\in \Grid}\sum_{\alpha=1}^{2}\sum_{\beta=1}^{2}\left(\uE^{n}_{\beta}\right)^{2}\sum_{\beta' = 1}^{2} \left( D_{\beta'}^{0}\ue^{n}_{\alpha} \right)^2 \frac{h^2}{L^2} \\
&\leq 4 \MDu^2 \|\uEb\|^2
\end{align*}

We may also treat the IB forcing term, by appealing to the boundedness and Lipschitz continuity of $\F$ and $\delta_{c}$.

\begin{equation}
\begin{split}
    \frac{1}{\rho} \left\|(\F(\Xcb^{n})-\F(\Xeb^{n}))\dc(\x-\Xcb^{n})\right\| + \frac{1}{\rho}\left\|\F(\Xeb^{n})(\dc(\x-\Xcb^{n})-\dc(\x-\Xeb^{n}))\right\| \\
    \leq \frac{\LF \Md + \MF \Ld}{\rho}\|\XEb^{n}\|
\end{split}
\end{equation}

This yields the inequality

\begin{equation}
\begin{split}
\|\usEb^{n+1}\|^2 + &\dt\frac{\mu}{\rho}\sum_{\alpha=1}^{2}\|D^{-}_{\alpha}\usEb^{n+1}\|^{2} \leq \\ &\|\usEb^{n+1}\|\Bigg(\|\usEb^{n}\| + \sqrt{2}\dt \Muc\left(\sum_{\beta=1}^{2}\|D^{-}_{\beta}\uEb^{n}\|^{2} \right)^{1/2}  \\
&\qquad + 2 \dt \MDu \|\uEb^{n}\| + \frac{\LF \Md + \MF \Ld}{\rho}\|\XEb^{n}\| - \frac{\dt}{\rho}\|\ures^{n}\| \Bigg)
\end{split}
\end{equation}

To isolate the factor $\|\usEb^{n+1}\|$ we apply the inequality $ab \leq \frac{1}{2}a^2 + \frac{1}{2}b^{2}$

\begin{equation}
\begin{split}
&\|\usEb^{n+1}\|^2 + \dt\frac{\mu}{\rho}\sum_{\alpha=1}^{2}\|D^{-}_{\alpha}\usEb^{n+1}\|^{2} \leq \frac{1}{2}\|\usEb^{n+1}\|^2  \\
&\qquad \qquad+ \frac{1}{2} \Bigg[\|\usEb^{n}\| + \sqrt{2}\dt \Muc \left(\sum_{\beta=1}^{2}\|D^{-}_{\beta}\uEb^{n}\|_{2}^{2} \right)^{1/2} + 2\dt \MDu \|\uEb^{n}\| \\
& \qquad\qquad\qquad + \dt\frac{\LF \Md + \MF \Ld}{\rho}\|\XEb^{n}\| + \frac{\dt}{\rho}\|\ures^{n}\| \Bigg)\Bigg]^{2}
\end{split}
\end{equation}

and subtract by $\frac{1}{2}\|\usEb^{n+1}\|^2$. Now that all instances of the $(n+1)$th step have been moved to the right hand side of the inequality, the next step is to find the growth factor. Rewriting the left hand side, we have

\begin{equation*}
\begin{split}
&\frac{1}{2}\|\usEb^{n+1}\|^2 + \dt\frac{\mu}{\rho}\sum_{\alpha=1}^{2}\|D^{-}_{\alpha}\usEb^{n+1}\|^{2} \leq  \\
&\quad
\frac{1}{2}
\left[ 
\left(
\begin{array}{c}
1 \\
2\sqrt{\dt T_{0}}\MDu \\
\sqrt{2\dt\frac{\rho}{\mu}}\Muc \\
\sqrt{\frac{\dt T_{0}^{3}}{2}}\left(\frac{\LF \Md + \MF \Ld}{\rho}\right) \\
\sqrt{\frac{\dt}{T_{0}}}
\end{array}
\right)
\cdot
\left(
\begin{array}{c}
\|\usEb^n\| \\
\sqrt{\frac{\dt}{T_{0}}}\|\uEb^n\| \\
\sqrt{2\dt\frac{\mu}{\rho}}\left(\sum_{\beta=1}^{2}\|D^{-}_{\beta}\uEb^{n}\|_{2}^{2} \right)^{1/2} \\
\sqrt{2\frac{\dt}{T_{0}^3}}\|\XEb^n\| \\
\frac{\sqrt{\dt T_{0}}}{\rho}\|\ures^{n}\|
\end{array}
\right)
\right]^{2}
\end{split}
\end{equation*}
Holds for any positive constant $T_{0}$. The introduction of this constant ensures each term in the first factor is dimensionless. Now we apply Cauchy-Schwarz to this inner product to obtain

\begin{equation}
\begin{split}
&\frac{1}{2}\|\usEb^{n+1}\|^2 + \dt\frac{\mu}{\rho}\sum_{\alpha=1}^{2}\|D^{-}_{\alpha}\usEb^{n+1}\|^{2} \leq  \\
&\quad \frac{1}{2}\left( 1 + \dt \underbrace{\left(4T_{0} \MDu^2 + \frac{\rho}{\mu}\Muc^2 + \frac{\LF \Md + \MF \Ld}{2\rho^2}T_{0}^{3} + \frac{1}{T_{0}} \right)}_{\text{constant}}\right) \\
&\quad\cdot \left( \|\usEb^{n}\|^2 + \frac{\dt}{T_{0}}\|\uEb^{n}\|^2 + 2\dt \frac{\mu}{\rho}\sum_{\alpha=1}^{2}\|D^{-}_{\beta}\uEb^{n}\|_{2}^{2} + \frac{2\dt}{T_{0}^3}\|\XEb^{n}\|^2 + \frac{\dt T_{0}}{\rho^2}\|\ures^{n}\|\right)
\end{split} \label{ErrorIneq}
\end{equation}

We now wish to address the remaining $\uEb^n$ terms. By the discrete Helmholtz decomposition (Lemma \ref{Decomp}), $\uEb^{n}$ may be decomposed as

\begin{equation}
    \uEb^{n} = \usEb^{n} + \D\psi^{n}
\end{equation}

And since this decomposition is orthogonal, we have

\begin{align*}
    \|\uEb^{n}\|^2 = \|\usEb^{n}\|^2 + \|\D\psi^{n}\|^2 \\
    \leq \|\usEb^{n}\|^2 + \frac{1}{\lambda}\|\pres^{n}\|^2
\end{align*}

Since 
Where the inequality comes from Lemma \ref{ConstraintBound}. Furthermore, since $\D$ and $D^{-}_{\alpha}$ are commutative,

\begin{align*}
    D^{-}_{\alpha}\uEb^{n} &= D^{-}_{\alpha}\usEb^{n} + \D D^{-}_{\alpha}\psi^{n} \\
    \D\cdot D^{-}_{\alpha} \usEb^{n} &= 0
\end{align*}

thus $D^{-}_{\alpha} \usEb^{n}$ and $\D D^{-}_{\alpha}\psi^{n}$ are orthogonal. Applying Lemma \ref{ConstraintBound} again, we find

\begin{align*}
    \|D^{-}_{\alpha}\uEb^{n}\|^2 = \|D^{-}_{\alpha}\usEb^{n}\|^2 + \|\D D^{-}_{\alpha}\psi^{n}\|^2 \\
    \leq \|D^{-}_{\alpha}\usEb^{n}\|^2 + \frac{1}{\lambda}\|D^{-}_{\alpha}\pres^{n}\|^2
\end{align*}

applying these decomposition inequalities to \eqref{ErrorIneq}, we have 

\begin{equation}
\begin{split}
&\frac{1}{2}\|\usEb^{n+1}\|^2 + \dt\frac{\mu}{\rho}\sum_{\alpha=1}^{2}\|D^{-}_{\alpha}\usEb^{n+1}\|^{2} \leq  \\
&\qquad \frac{1}{2}\left( 1 + O(\dt) \right)\cdot \left( \left(1 + \frac{\dt}{T_{0}}\right) \left(\frac{1}{2}\|\usEb^{n}\|^2 + \dt\frac{\mu}{\rho}\sum_{\alpha=1}^{2}\|D^{-}_{\alpha}\usEb^{n}\|^{2} \right)\right. \\
&\qquad+ \left.\frac{\dt}{T_{0}^3}\|\XEb^{n}\|^2 + \dt\left(\frac{\|\pres^{n}\|}{2\lambda T_{0}} + \frac{\mu}{\rho\lambda}\sum_{\alpha=1}^{2}\|D^{-}_{\alpha}\pres^{n}\|^{2} +\frac{\dt T_{0}}{2\rho^2}\|\ures^{n}\|\right)\right)
\end{split} 
\end{equation}

The second inequality (for $\XEb^{n+1}$) is obtained through a similar procedure to the one above. From \eqref{ErrIBP1}, we take the inner product with $\XEb^{n+1}$, and find

\begin{equation}
\begin{split}
&\|\XEb^{n+1}\|^{2} = \left(\XEb^{n+1},\XEb^{n}\right) + \dt \left(\XEb^{n+1},\sum_{\x\in\Grid}\uEb^{n}(\x)\dc(\x-\Xcb^{n})\dx^2\right) \\
&\quad+ \dt\left(\XEb^{n+1},\sum_{\x\in\Grid}\ueb^{n}(\x)\left(\dc(\x-\Xcb^{n})-\dc(\x-\Xeb^{n})\right)\dx^2\right) -\dt (\XEb^{n+1},\Xres^{n})
\end{split}
\end{equation}

From Cauchy-Schwarz, we obtain 

\begin{equation}
\begin{split}
&\|\XEb^{n+1}\|^{2} \leq \|\XEb^{n+1}\| \left(\|\XEb^{n}\| + \dt\left\|\sum_{\x\in\Grid}\uEb^{n}(\x)\dc(\x-\Xcb^{n})\dx^2\right\|\right. \\
&\qquad\left.+\dt \left\|\sum_{\x\in\Grid}\ueb^{n}(\x)\left(\dc(\x-\Xcb^{n})-\dc(\x-\Xeb^{n})\right)\dx^2\right\| + \dt \|\XEb^{n+1}\| \right)
\end{split}
\end{equation}

Since

\begin{align*}
    &\left\|\sum_{\x\in\Grid}\uEb^{n}(\x)\dc(\x-\Xcb^{n})\dx^2\right\|^2 = \sum_{\alpha=1}^{2}\left(\sum_{\x\in\Grid}\uEb^{n}(\x)\dc(\x-\Xcb^{n})\dx^2\right)^2 \\
    &\qquad\qquad= \length^4 \sum_{\alpha=1}^{2}\left(\sum_{\x\in\Grid}\uEb^{n}(\x)\dc(\x-\Xcb^{n})\frac{\dx^2}{\length^2}\right)^2 \\
    &\qquad\qquad\leq \length^4 \left(\sum_{\alpha=1}^{2}\left(\sum_{\x\in\Grid}\left(\uEb^{n}(\x)\right)^2\frac{\dx^2}{\length^2}\right)\left(\sum_{\x'\in\Grid}\dc^{2}(\x'-\Xcb^{n})\frac{\dx^2}{\length^2}\right)\right)^2 \\
    &\qquad\qquad \leq \length^4 \|\uEb^{n}\|^2 \Md^2
\end{align*}

and

\begin{equation}
\begin{split}
    &\left\|\sum_{\x\in\Grid}\ueb^{n}(\x)\left(\dc(\x-\Xcb^{n})-\dc(\x-\Xeb^{n})\right)\dx^2\right\|^2 \\
    &\qquad\qquad= \sum_{\alpha=1}^{2}\left(\sum_{\x\in\Grid}\ueb^{n}(\x)\left(\dc(\x-\Xcb^{n})-\dc(\x-\Xeb^{n})\right)\dx^2\right)^2 \\
    &\qquad\qquad \leq 2\length^4 \Mu^2 \Ld \|\XEb^{n}\|^2
    \end{split}
\end{equation}

we infer

\begin{align}
\|\XEb^{n+1}\|^{2} &\leq 
\left[ 
\left(
\begin{array}{c}
1 \\
\sqrt{2\dt T_{0}} \length^2 \Mu \Ld\\
\sqrt{\frac{\dt}{T_{0}}}\length^2 \Md \\
\sqrt{\frac{\dt}{T_{0}}}
\end{array}
\right)
\cdot
\left(
\begin{array}{c}
\|\XEb^n\| \\
\sqrt{\frac{\dt}{T_{0}}}\|\XEb^n\| \\
\sqrt{\dt T_{0}} \|\uEb^{n}\| \\
\sqrt{\dt T_{0}}\|\Xres^n\|
\end{array}
\right)
\right]^{2}
\end{align}

so from Cauchy-Schwarz,

\begin{align*}
\begin{split}
    &\|\XEb^{n+1}\|^{2} \leq \Bigg(1 +\dt\underbrace{\left(2\Mu^2 L_{\delta_{c}}^2 L^4 T_{0} + L^4\frac{1}{T_{0}} \Md^2 + \frac{1}{T_{0}}\right)}_{\mathrm{constant}}\Bigg)\cdot \\
    &\quad \qquad \left(\|\XEb^{n}\|^2 + \frac{\dt}{T_{0}}\|\XEb^{n}\|^2 + \dt T_{0}\|\uEb^{n}\|^2 + \dt T_{0}\|\Xres^{n}\|^2\right)
\end{split}
 \\
&\leq (1 +O(\dt))\left(\left( 1 + \frac{\dt}{T_{0}}\right)\|\XEb^{n}\|^2 + \dt T_{0} \left(\|\usEb^{n}\|^2 + \frac{1}{\lambda}\|\pres^{n}\|^2+ \|\Xres^{n}\|^2\right)\right) \\
\begin{split}
&\leq(1 +O(\dt))\left(\left( 1 + \frac{\dt}{T_{0}}\right)\|\XEb^{n}\|^2 + 2\dt T_{0} \left( \frac{1}{2}\|\usEb^{n}\|^2 + \dt\frac{\mu}{\rho}\sum_{\alpha=1}^{2}\|D^{-}_{\alpha}\usEb^{n}\|^{2} \right)\right. \\
& \qquad\qquad+ \left.\dt T_{0} \left(\frac{1}{\lambda}\|\pres^{n}\|^2+ \|\Xres^{n}\|^2\right)\right)
\end{split}
\end{align*}
Now, defining
\begin{align*}
A^{n} &= \frac{1}{2}\|\usEb^{n}\|^2 + \dt\frac{\mu}{\rho}\sum_{\alpha=1}^{2}\|D^{-}_{\alpha}\usEb^{n}\|^{2} \\
B^{n} &= \frac{1}{T_{0}^2}\|\XEb^{n}\|^2
\end{align*}

and substituting into the inequalities above we obtain the the desired result.
\end{proof}

\section{Simulation}
\label{sec:sim}
In this section we run a simulation and confirm empirically the theoretical result from the preceding sections. For this experiment, we chose to simulate the formation of a K\'{a}rm\'{a}n vortex street in the wake of a flow driven past a cylinder tethered to a fixed point via a Hookean spring. The stiffness of the spring is chosen to allow the cylinder to move substantially in order to test convergence of the numerical method in the context of a moving boundary. The continuous equations we shall use to model this phenomenon are

\begin{align}
 \rho\left(\frac{\partial\ueb}{\partial t} + (\ueb\cdot\nabla)\ueb \right) &= \mu\nabla^2 \ueb - \nabla \pe + \feb(\x,t) + a(t)\mathbf{e}_{1} \label{ContMomExp} \\
\nabla\cdot\ueb & = 0 \label{ContIncompExp} \\
\feb(\x,t) &= -k(\Xeb(t)-\Xeb_{0})\dc(\x-\Xeb(t)) \label{ContFDExp}\\
\frac{d\Xeb}{dt}(t) & =  \int_{\Dom} \ueb(\x,t)\dc(\x-\Xeb)d\x \label{ContInterpExp}  \\
a(t) &= -\int_{\Omega}\mathbf{e}_{1}\cdot \feb(\x,t) d\x \label{meanflowforce}
\end{align}
where $\mathbf{e}_1,\mathbf{e}_2$ are the unit basis of the coordinate axes. This system of equations is similar to equations \eqref{ContMom}-\eqref{ContInterp} of section \ref{ibproblem} with the specific choice of forcing function: $\F(\Xeb) = -k(\Xeb - \Xeb_0)$ for some constant $k$, and the addition of a driving body force $a(t)\mathbf{e}_{1}$. The function $a(t)$ is chosen specifically such that the mean flow along the $x$-axis remains constant over time, and is therefore determined by its initial value.

In the above equations, the cylinder is modelled in very simple way.
The variable $\Xeb(t)$ keeps track of its centre, and the
\textit{footprint} of the cylinder, so to speak, is modelled by the
function $\dc(\x - \Xeb(t))$, which is used in the application of
force to the fluid, see equation \ref{ContFDExp}, and also in the evaluation of the velocity with which the cylinder is moving, see equation \ref{ContInterpExp}. We are concerned here not with the validity of this model, but rather with the convergence of the numerical method that we use to solve the model equations.

\subsection{The delta approximation $\dc$}

Before proceeding to the implementation of the algorithm, we must give a specific choice of $\dc$ as the IB method does not specify any particular choice of delta kernel. We note that for the purpose of the theory, we have only required that $\dc$ should be twice continuously differentiable and that its integral should equal to 1. With regard to differentiability, note that typical IB kernels \cite{peskin2002immersed} are only once continuously differentiable, so these cannot be used for our present purpose.  The kernel that we actually use will have three continuous derivatives -- one more than needed for
the theory. This subsection aims to motivate and discuss the specific choice of IB kernel used in this simulation, which was originally discovered and advocated by Bao, Kaye and Peskin \cite{bao2016gaussian}. To the two conditions mentioned above, we will add supplementary conditions until the kernel function is uniquely defined. Firstly, we impose the following form

\begin{equation}
\dc(\x) = \frac{1}{c^2}\varphi\left(\frac{x_{1}}{c}\right)\varphi\left(\frac{x_{2}}{c}\right) \label{DeltaDef}
\end{equation}

For a general choice of $\varphi$ this condition breaks radial symmetry. However, the condition allows $\dc$ to satisfy favourable properties on the uniform Cartesian grid $\Grid$; furthermore it turns out that we will be able to chose a $\varphi$ that is near Gaussian, making $\dc$ nearly radially symmetric. The physical parameter $c$ scales the dimensionless function $\varphi$ to the width of the immersed particle. We also require that $c$ be an integer multiple of the gridwidth; again, this restriction is chosen for the beneficial properties it confers when passing from the continuous to the discrete equations. To find $\varphi$, we look for a function with the following properties

\begin{itemize}
    \item $\varphi$ should have a bounded support. When computing the discretised analogue of the integral in \eqref{ContInterpExp}, it would be best for algorithmic efficiency to sum the integrand over the portion of the mesh that is in the vicinity of the immersed particle, rather than over the entire fluid mesh.
    \item Ideally, it should matter as little as possible where the centre of the cylinder is in relation to the fluid mesh. In other words, for any function $\psi$, we should have for all real $r$, the important point being that $r$ is not necessarily integer-valued, that
    
    \begin{equation}
        \sum_{i\in\mathbb{Z}}\psi(r-i)\varphi(r-i) \quad \text{is constant} \label{transInvariance}
    \end{equation}

    Unfortunately, this is too much to ask, and no such $\varphi$ exists\footnote{If we allow $\phi$ to have unbounded support, and restrict $\psi$ to be band limited with suitably chosen bandwidth, then it is interesting to note that equation \eqref{transInvariance} can indeed be satisfied \cite{chen2024fourier}.}. So instead, we may require the above invariance to hold for only certain choices of $\psi$ i.e. $\psi(x)\in\{1,x,x^2,x^3,\varphi(x)\}$
\end{itemize}

These requirements suggest the following conditions on $\varphi$:
 
\begin{enumerate}
	\item The function $\varphi$ has a continuous third derivative.
	\item The support of $\varphi$ is $(-3,3)$.
	\item For all real $r$, $\varphi$ satisfies the following discretised moment equations
	\begin{align}
	&\sum_{\substack{j\in\mathbb{Z}\\ j\;\mathrm{even}}} \varphi(r-j) = \sum_{\substack{j\in\mathbb{Z}\\ j\;\mathrm{even}}} \varphi(r-j) = 1/2\\
	&\sum_{j\in \mathbb{Z}} (r-j)\varphi(r-j) = 0 \\
	&\sum_{j\in \mathbb{Z}} (r-j)^2\varphi(r-j) = K \label{phi2mom}\\
	&\sum_{j\in \mathbb{Z}} (r-j)^3\varphi(r-j) = 0
	\end{align}
	\item For all real $r$, $\varphi$ satisfies the non-linear condition
	\begin{equation}
	\sum_{j\in\mathbb{Z}}\varphi^2(r-j) = C
	\end{equation}
\end{enumerate}
where $C$ and $K$ are parameters independent of $r$ that remain to be determined. From conditions 2-4, it is possible to find a continuous set of continuously differentiable functions---parametrised by $K$. For each $K$, there is a unique corresponding value of $C$ (though the explicit relation between the two constants is not simple, see \cite{bao2016gaussian} for details). Of the 1 parameter family, there is, in fact, a unique choice

\begin{equation}
K = \frac{59}{60}-\frac{\sqrt{29}}{20}
\end{equation}

such that $\varphi$ has three continuous derivatives. For This choice of $K$, $C$ is approximately $0.326$. The paper cited above gives a derivation of $\varphi$ and its explicit algebraic expression. The unique function $\varphi$ that satisfies all of the above conditions is plotted in \ref{fig:delta}

\begin{figure}[htbp]
  \centering
  \label{fig:delta}\includegraphics[width=0.6\linewidth]{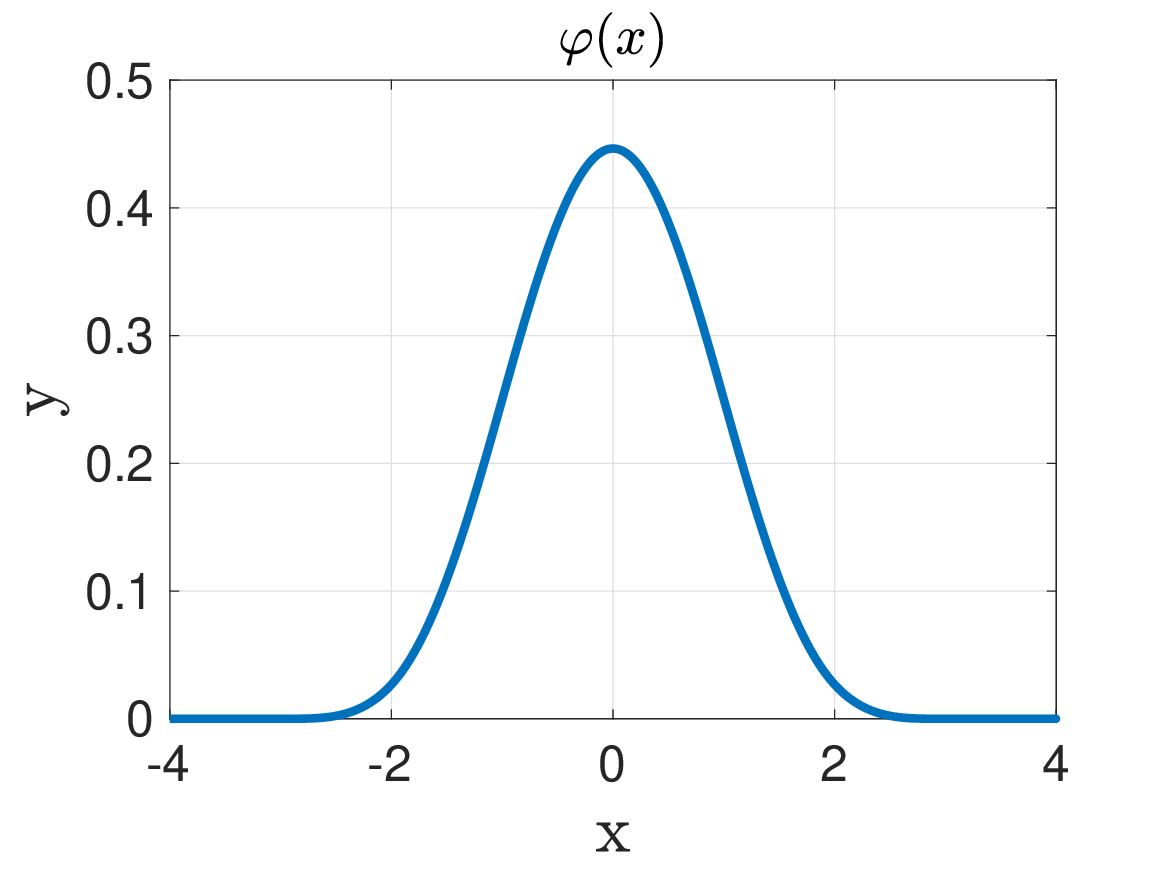}
  \caption{The function $\varphi$ defined by conditions 1-5. The support of the function is defined over a closed set--- the closure of the set of points at which the function is non-zero. This is useful here because +3 and -3 are therefore within the support and they are indeed the points at which $\varphi$ is only $C^{3}$}
  \label{fig:testfig}
\end{figure}

\begin{figure}[th]
    \centering
    \includegraphics[width=\linewidth]{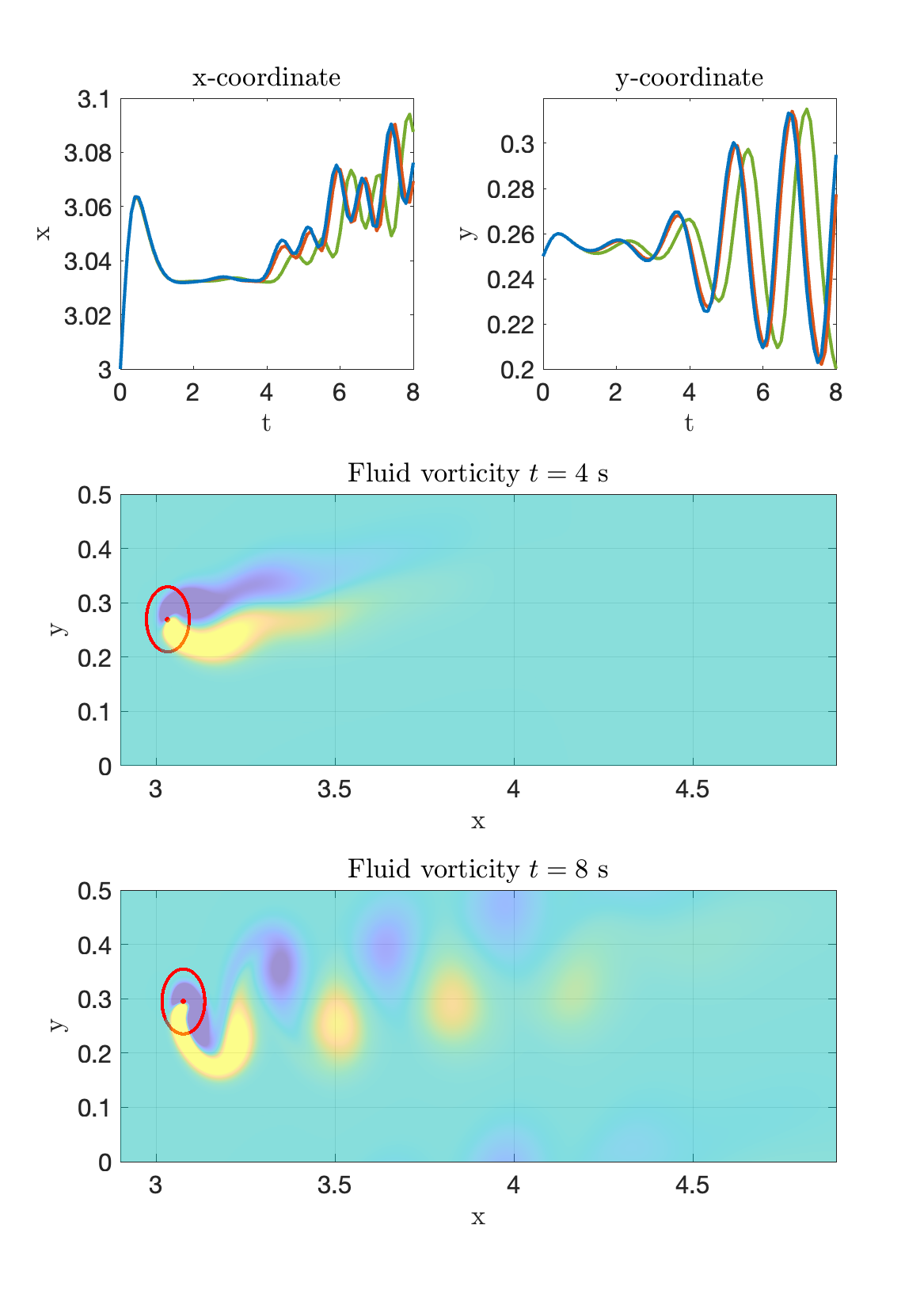}
    \caption{Results of a numerical experiment. Top: the x and y coordinates of the centre of the cylinder as a function of time. The different colours represent the analogous results for various grid refinements, with yellow representing the coarsest mesh and blue representing the finest grid. The cylinder’s x coordinate appears to reach a an equilibrium as the mean flow balances with the restoring spring force. This equilibrium breaks down at about t= 4s. This timing coincides with the moment that the vortex sheet in the wake of the cylinder becomes unstable, breaking into a vortex street. This vortex street then induces a force that pulls the cylinder further from the centre. Bottom: A colourmap of the fluid vorticity. Note the upward trend of the vortex street. This is because of the choice of initial fluid velocity (chosen to break symmetry).}
    \label{fig:snap}
\end{figure}

Although our model cylinder has no sharp edge, it is useful to think
of it as having an effective radius $R$ defined by

\begin{equation}
    R^2 = \sum_{\x\in\Grid} |\x - \Xeb(t)|^2 \delta_c(\x - \Xeb(t)) h^2 \label{DeltaDisc2mom}
\end{equation}

Thus $R$ is the root mean squared distance of all grid points from the
centre of the cylinder, with \textit{mean} defined by the normalized
weights $\dc(\x-\Xeb(t))h^2$.  Although it might appear that this
definition of the radius depends on the location $\Xeb(t)$ of the centre of the
cylinder, this is not the case, and indeed we can deduce directly
from equations \eqref{DeltaDef} and \eqref{phi2mom} that

\begin{equation}
    R^2 = 2Kc^2
\end{equation}

and this is not only independent of location but also independent of
the gridwidth h.\footnote{The IB kernel described here was created in
response to a request from Aleks Donev for a kernel with a constant
positive second moment, so that such a kernel could be used in the
representation of a particle of constant effective radius,
independent of the location of that particle with respect to a
computational mesh.}  It may also be worth mentioning that

\begin{equation}
    R^2 = \int_{\Omega}|\x-\y|^2\dc(\x-\y)d\x \label{Delta2mom}
\end{equation}

and we leave it as an interesting exercise for the reader to show
that the integral in \eqref{Delta2mom} is \textit{equal} to the sum in \eqref{DeltaDisc2mom}.

\subsection{Parameters of the Experiment}
For this simulation, we define the domain $\Omega$ to be a rectangle of 6 m in length by 0.5 m in width. Therefore, the centre of the domain is $\Xeb_{0}=(0.25 \text{ m},3\text{ m})$. The initial conditions were chosen as follows
\begin{align}
    \Xeb(0) &= \Xeb_{0} \\
    u_{1}(\x,0) &= u_{\text{mean}} \\
    u_{2}(\x,0) &= v_{0}
\end{align}

Note that because of the choice of $a(t)$, $u_{\text{mean}}$ remains equal to the mean flow along $x$-axis for all time. We also choose a non-zero initial vertical velocity $v_{0}$ to break the symmetry, and allow the discretisation errors to dominate the round-off errors. In the following table, we list the values of the physical parameters used in the simulation

\begin{table}[htbp]
{\footnotesize
  \caption{Listed here are the physical constants used in the experiment. Alongside with the heuristic definition of radius given by \eqref{Delta2mom}, we can estimate the Reynolds number at a value of 150}  \label{tab:foo}
\begin{center}
  \begin{tabular}{|c|l|}
    \hline
    Parameter & Value \\
    \hline
         $\rho$ & $1.00 \;\mathrm{kg}/\mathrm{m}^3$   \\
         $\mu$ & $4.00E-4 \;\mathrm{kg}/\mathrm{m}\cdot\mathrm{s}$ \\
         $k$ & $1.00E-1\;\mathrm{N}/\mathrm{m}^2$ \\
         $u_{\text{mean}}$ & $2.50E-1\;\mathrm{m}/\mathrm{s}$ \\
         $v_{0}$ & $4.00E-2\;\mathrm{m}/\mathrm{s}$\\
         $c$ & $1.00E-1\;\mathrm{m}$ \\
         \hline
             \end{tabular}
\end{center}
}
\end{table}
The units are obtained by assuming the fluid flow is perpendicular to the cylinder and symmetric under translations parallel to the cylinder. From these units, and from the formula $R^2 = 2Kc^2$, we can calculate the Reynolds number:

\begin{equation}
    \mathrm{Re} = 150
\end{equation}

According to \cite{anagnostopoulos1992response} we should expect to see laminar Vortex shedding for this choice of Reynold's number.

\subsection{Numerical implementation}
The numerical method employed here for solving the problem stated above is nearly identical to the method given in section \ref{algorithm}, the only change being that for each timestep the zeroth Fourier mode of the $x_{1}$ component of the fluid velocity is prescribed to equal the desired mean flow rate of $u_{\text{mean}}$. Note that this is different from evaluating $a(t)$ explicitly via formula \eqref{meanflowforce} at each timestep. One may think of this method as analogous to Chorin's projection method---at each timestep the fluid velocity is projected onto both the incompressibility {\it and} mean flow constraints. These changes from algorithm \ref{alg:thy} yields the algorithm \ref{alg:exp}:
\begin{algorithm}
\caption{Immersed Boundary Method}
\label{alg:exp}
\begin{algorithmic}
\While{$n < T/\dt$}
\State{Compute the term $\wcb^{n} = \ucb^{n}-\dt\ucb^{n}\cdot\D\ucb^{n} + \frac{\dt}{\rho}\F(\Xcb^{n})\dc(\x-\Xcb^{n})$}
\State{Set $\hat{\wc}^{n}_{1}(0) = u_{\rm{mean}}$} \Comment{$\hat{\wcb}^{n}(0)$ refers to the zeroth discrete Fourier mode}
\State{Solve $\begin{cases}
    \left(I-\dt\frac{\mu}{\rho}\Dp\cdot\Dm\right)\ucb^{n+1} = \wcb^{n} - \frac{\dt}{\rho}\D \pc^{n+1} \\
    \D \cdot \ucb^{n+1} = 0
\end{cases}$ for $\ucb^{n+1}$}
\State{Calculate $\Xcb^{n+1} = \Xcb^{n} + \dt \sum_{\x\in \Grid}\ucb^{n}(\x)\dc(\x-\Xcb^{n})\dx^2$}
\State{$n \leftarrow n+1$}
\EndWhile \\
\Return $\ucb^{n},\Xcb^{n}$
\end{algorithmic}
\end{algorithm}

Fig \ref{fig:snap} illustrates the result of the simulation outlined above up to a time $T=8s$. Note that at time $t=4s$ the instability of the symmetrical flow pattern becomes apparent. As we shall see, this instability does not  prevent convergence of the numerical scheme.

Algorithm \ref{alg:exp} was run multiple times on different resolutions. The finest resolution was on a mesh with gridwidth $\dx=0.00104$ m and a timestep size of $\dt =1.95 E-5$ s whereas the coarsest simulation was done on a gridwidth $\dx = 0.00833$ m and a timestep size of $1.25E-3$ s. As the exact solution is unknown, it is necessary to use the difference ins the output between successive simulations to estimate the order of convergence. This estimate was done in the following way: A grid $\Grid$ is selected and the results $\ucb^{T/\dt}_{\rm{coarse}},\Xcb^{T/\dt}_{\rm{coarse}}$ of the simulation are stored. The grid is then coarsened to $G_{2h}$ and the timestep size is quadrupled. The simulations are repeated in this lower resolution; the results $\ucb^{T/4\dt},\Xcb^{T/4\dt}$ are stored. The fluid velocity results are compared by measuring the $L_{2}$ norm of their difference
\begin{equation}
\left(\sum_{\x\in G_{2h}}\left\|\ucb^{T/\dt}(\x) - \ucb^{T/\dt}_{\rm{coarse}}(\x)\right\|^{2} \frac{h^2}{L^2}\right)^{\frac{1}{2}}
\end{equation}
Whereas the difference in the cylinder positions is simply the Euclidean distance
\begin{equation}
    \left\|\Xcb^{T/\dt} - \Xcb^{T/\dt}_{\rm{coarse}}\right\|
\end{equation}

It is well known \cite{leveque2007finite} that these differences should have the same asymptotic behaviour as the unobserved errors, so our prediction is that the the differences between solutions computed on successive meshes should be $O(\dx^2)$. This hypothesis is confirmed by the data in figure \ref{errplt}.

\begin{figure}[h!]
    \centering
    \includegraphics[width=0.65\linewidth]{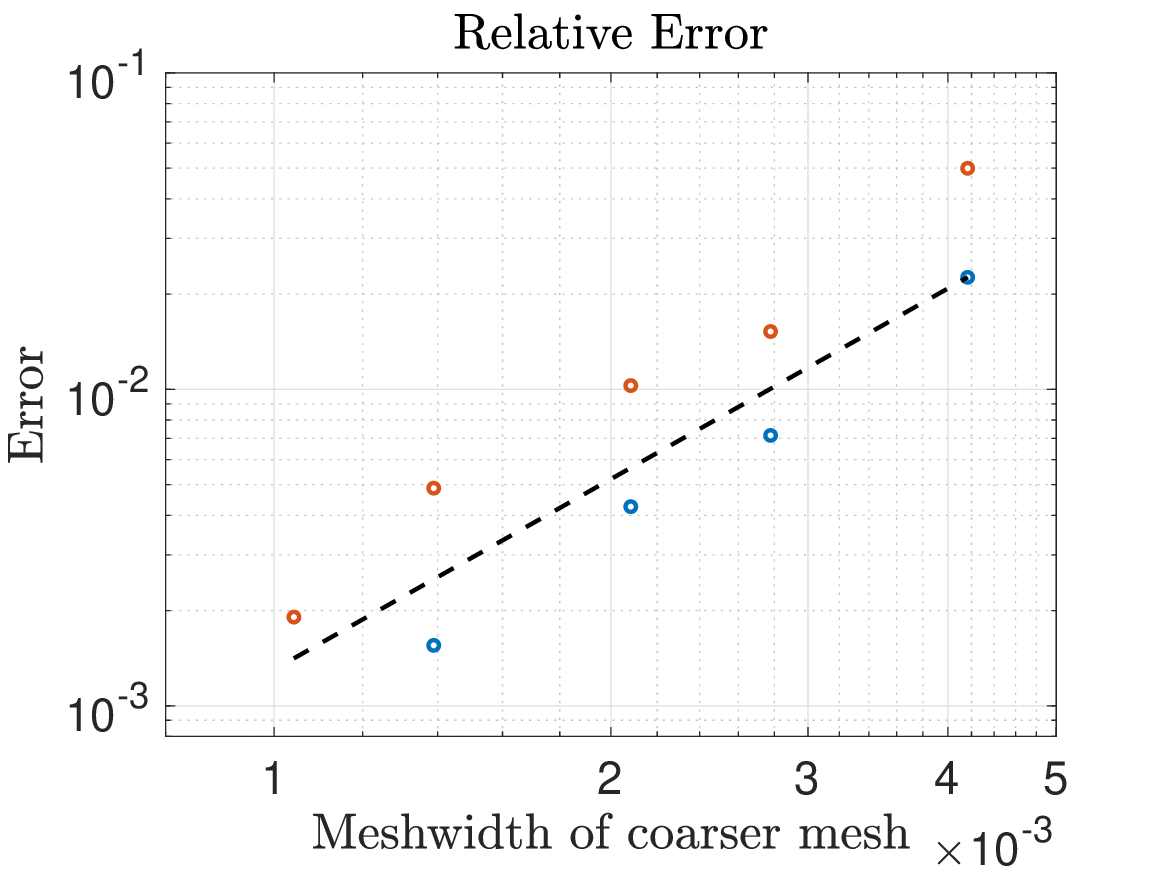}
    \caption{These data show the results of an empirical convergence analysis of the experiment described above. Each measurement is made by comparing a simulation to the same simulation on a coarser grid (double the gridwidth and quadruple the timestep size). Both simulations were run to the final time $T=8$ s and compared. The norm of the difference is then plotted on this graph. The blue data point is the discrete $L_2$ norm of the difference in the fluid velocity between the coarse and fine grids. The red data is the distance between the centres of the cylinders for the fine and coarse simulations at the final time.}
    \label{errplt}
\end{figure}

\section{Conclusion}

In practice convergence of the IB method needs to be confirmed empirically on a case by case basis. This is because current understanding of the convergence properties is limited, especially in the context of a non-linear fluid. The lack of theoretical results is due to the presence of singular forcing terms in the governing equations. Yoichiro Mori was the first to break through this complication by using linear fluid equations (Stokes flow), and by considering the velocity field generated by given boundary forces. The idea behind the present paper is different: we instead retain the non-linear elements of the problem and resolve the regularity issue by mollifying the delta function in the governing equations. This was originally inspired by physical necessity for the case in which the co-dimension of the immersed material is greater than one, but the same approach may be technically useful as a step in a convergence proof even when it is not physically required. This will be the subject of future investigation.

\pagenumbering{gobble}
\end{document}